\documentclass[12pt]{amsart}

\usepackage{tikz-cd}
\usetikzlibrary{arrows}
\tikzset{
   commutative diagrams/.cd,
   arrow style=tikz,
   diagrams={>=latex}}
\usepackage{amsmath,amssymb}

\newtheorem{theorem}{Theorem}[section]
\newtheorem{lemma}[theorem]{Lemma}

\newtheorem{corollary}[theorem]{Corollary}

\theoremstyle{definition}

\newtheorem{remark}[theorem]{Remark}

\title{Coverings of torus knots in $S^2\times S^1$ and universals}

\date{}

\author{V\'{\i}ctor N\'unez}
\email{victor@cimat.mx}
\address{CIMAT, A.P. 402, Guanajuato 36000, Gto., M\'EXICO}

\author[Enrique Ram\'{\i}rez]{Enrique Ram\'{\i}rez-Losada}
\email{kikis@cimat.mx}
\address{CIMAT, A.P. 402, Guanajuato 36000, Gto., M\'EXICO}

\author[Jes\'us Rodr\'{\i}guez]{Jes\'us Rodr\'{\i}guez-Viorato}
\email{jesusr@cimat.mx}
\address{CIMAT, A.P. 402, Guanajuato 36000, Gto., M\'EXICO}


\begin{document}
\maketitle

\begin{abstract}
Let $t_{\alpha,\beta}\subset S^2\times S^1$ be an ordinary fiber of a
Seifert fibering of $S^2\times S^1$ with two exceptional fibers of order
$\alpha$.
We show that any Seifert manifold with Euler number zero is a branched
covering of $S^2\times S^1$ with branching $t_{\alpha,\beta}$ if $\alpha\geq3$.

We compute the Seifert invariants of the Abelian covers of $S^2\times
S^1$ branched along a $t_{\alpha,\beta}$.

We also show that $t_{2,1}$, a non-trivial torus knot in $S^2\times S^1$, is
not universal.
\end{abstract}

\section{Introduction.}

The study of branched coverings of Seifert manifolds branched along fibers is an interesting project (\cite{seifert}). These kind of coverings are also Seifert manifolds. In particular the problem of determining which Seifert manifolds are branched coverings with branching along a torus knot is interesting, where a \emph{torus knot in $M$} is an ordinary fiber of a Seifert fibering of~$M$. 

Recall that in a branched covering of Seifert manifolds, the Euler number of the covering is a non-zero rational multiple of the Euler number of the target. Therefore in a branched covering $M\rightarrow N$ between Seifert manifolds, if the Euler number $e(N)\neq0$, then $e(M)\neq0$.
In~\cite{trefoil}, a converse of this implication is proven for torus knots in the 3-sphere: If $M$ is an orientable Seifert manifold with orientable basis and Euler number $e(M)\neq0$, and $\tau_{p,q}\subset S^3$ is a torus knot with $2\leq q<p$, then there is a branched covering $\varphi:M\rightarrow(S^3,\tau_{p,q})$. 

What about if $e(M)=0$? In this paper we answer this question: If $M$ is an orientable Seifert manifold with orientable basis and Euler number $e(M)=0$, and $t_{\alpha,\beta}\subset S^2\times S^1$ is a torus knot  (see Section~\ref{sec223}), with $3\leq\beta\leq\alpha/2$, then there is a branched covering $\varphi:M\rightarrow(S^2\times S^1,t_{\alpha,\beta})$.

Resuming, all orientable Seifert manifolds with orientable basis can be constructed as branched coverings branched along torus knots in $S^3$ or in $S^2\times S^1$.

In contrast with the result in \cite{trefoil}, that all non-trivial torus knots in~$S^3$ are universal in the sense above, we show that $t_{2,1}\subset S^2\times S^1$ is a non-trivial torus knot which is not universal (Theorem~\ref{thm47}).

This paper is organized as follows.

In Section~\ref{sec2} we give some definitions and list some useful results.

In Section~\ref{sec3} we give a list of the Abelian coverings of $(S^2\times S^1,t_{\alpha,\beta})$. It follows that if $\varphi:M\rightarrow(S^2\times S^1,t_{\alpha,\beta})$ is an Abelian branched covering and $H_1(M)\cong\mathbb{Z}$, then $M\cong S^2\times S^1$, and $\varphi$ is an unbranched cyclic covering space. Compare with the analogous result: If $M$ is an $n$-fold cyclic branched covering of $(S^3,\tau_{p,q})$, then $H_1(M)=0$ if and only if~$(n,pq)=1$.

In Section~\ref{sec4} we prove the main theorem of the paper (Corollary~\ref{coro46}).

\section{Preliminaries.}
\label{sec2}
Let $S_n$ be the symmetric group in $n$ symbols.
If $G\leq S_n$, and
$i\in \{1,\dots,n\}$, write $St_G(i)=\{\sigma\in G:\sigma(i)=i\}$; also write
$St(i)=St_{G}(i)$ when $G=S_n$. We write $(1)\in S_n$ for the identity
permutation. 

A {\sl branched covering} between two 
$n$-manifolds $M$ and $N$ is an open, proper map
~$\varphi:M\rightarrow N$ which is finite-to-one. To check that an open map 
$\varphi:M\rightarrow N$ is a branched covering, one finds a subcomplex 
$K\subset M$ of codimension two such that the restriction
~$\varphi|:M-\varphi^{-1}(K)\rightarrow N-K$ is a finite covering
space.  
The subcomplex $K$, usually a submanifold, is called the {\sl branching} 
of $\varphi$.
The covering space 
$\varphi|:M-\varphi^{-1}(K)\rightarrow N-K$ is called the {\sl
  associated covering space} of $\varphi$. The associated covering
determines the branched  
covering (see~\cite{Fox}).

A covering space of $n$ sheets $\psi:X\rightarrow Y$ defines a 
representation~$\omega_\psi:\pi_1(Y)\rightarrow S_n$ as follows: Number the elements
of the preimage 
of a base point $\psi^{-1}(*)=\{1,2,\ldots,n\}$, and for a class 
$[\alpha]\in\pi_1(Y)$, consider the liftings of $\alpha$,
$\alpha_1,\alpha_2,\ldots,\alpha_n:I\rightarrow X$ where $\alpha_i$ 
starts in the point~$i\in\psi^{-1}(*)$; then define 
$\omega_\psi([\alpha])(i)=\alpha_i(1)$, for $i=1,2,\ldots,n$.

A homomorphism $\omega:\pi_1(Y)\rightarrow S_n$ determines a 
covering space 
of $n$ sheets $\psi_\omega:X\rightarrow Y$, namely, the covering space 
corresponding to the subgroup $\omega^{-1}(St(1))\leq\pi_1(Y)$; this
covering space can be completed into a branched covering
$\psi_\omega:\bar X\rightarrow \bar Y$ (\cite{Fox}). A
representation of $\psi$, as~$\omega_\psi$ above, is conjugate to the 
homomorphism $\omega$ (it only depends on the numbering of 
$\psi^{-1}(*)$).

We will describe a branched covering of a manifold 
$M$ by giving a codimension two submanifold $K\subset M$ together with a 
representation~$\omega:\pi_1(M-K)\rightarrow S_n$. For short, we write $N\rightarrow(M,K)$ for a branched covering of $M$ branched along $K$.

\subsection{Seifert manifolds}
Let 
$\alpha_1,\beta_1,\alpha_2,\beta_2,\ldots,\alpha_t,\beta_t$ be 
integers such that $\alpha_i>0$ and $(\alpha_i,\beta_i)=1$ for 
$i=1,2,\ldots,t$. The {\sl Seifert manifold}~$M$ associated to the 
{\sl Seifert symbol }
$\displaystyle(Oo,g;{\beta_1/\alpha_1},{\beta_2/\alpha_2},\ldots,{\beta_t/\alpha_t})$ is assembled as follows:

Let $F$ be an oriented  closed surface of genus $g$, and let 
$D_1,D_2,\ldots,D_t\subset F$ be $t$ disjoint 2-disks. Write 
$F_0=\overline{F-\cup D_i}$, and $M_0=F_0\times S^1$. If $\partial 
F=q_1\sqcup q_2\sqcup\cdots\sqcup q_t$ and $h=\{x\}\times S^1$ for 
some $x\in F_0$, we let $m_i\subset q_i\times S^1$ be a simple closed curve 
such that $m_i\sim
q_i^{\alpha_i}h^{\beta_i}$. Let~$V_1,V_2,\ldots,V_t$ be solid tori
with meridians  
$\mu_1,\mu_2,\ldots,\mu_t$, respectively, and let $\eta_i:\partial 
V_i\rightarrow q_i\times S^1$ be a homeomorphism such that 
$\eta_i(\mu_i)=m_i$ for $i=1,2,\ldots,t$. Then the Seifert manifold 
$M$ associated to the
symbol~$\displaystyle(Oo,g;{\beta_1/\alpha_1},\ldots,{\beta_t/\alpha_t})$
is  
$$M=M_0\bigcup_{\cup\eta_i}\left(\bigcup_{i=1}^t V_i\right).$$

The circles $\{x\}\times S^1$ for $x\in F_0$ are called the {\sl ordinary 
fibers} of $M$ and the core $e_i$ of $V_i$ is called an {\sl
exceptional fiber of order $\alpha_i$} if $\alpha_i>1$; otherwise
$e_i$ is also an ordinary fiber. In any case $e_i$ is called \emph{the fiber
of the ratio $\beta_i/\alpha_i$} ($i=1,2,\dots,t$). The surface $F$ is
called the {\sl orbit surface} or the {\sl base} of $M$. Note that
collapsing each fiber of $M$ into a point gives an
identification~$p:M\rightarrow F$. 

The manifold $M$ associated to
$\displaystyle(Oo,g;
{\beta_1/\alpha_1},\dots,{\beta_t/\alpha_t})$ can be recovered
unambiguously from $M_0=F_0\times S^1$ and the
curves~$m_i=q_i^{\alpha_i}h^{\beta_i}$, $i=1,2,\dots,t$. We call the
pair 
$(F_0\times S^1,\{m_i\}_{i=1}^t)$ a {\sl frame} for~$\displaystyle(Oo,g;
{\beta_1/\alpha_1},\dots,{\beta_t/\alpha_t})$. 

Let $V$ be a solid torus. Any covering space of $\partial V$ extends
into a branched covering of $V$ branched at most along the core of $V$.
Therefore to describe a branched covering of 
 $\displaystyle(Oo,g; {\beta_1/\alpha_1},\dots,{\beta_t/\alpha_t})$
 branched along fibers, it suffices to construct a covering space of a frame for
$\displaystyle(Oo,g; {\beta_1/\alpha_1},\dots,{\beta_t/\alpha_t},
{0/1},\dots,{0/1})$.

The classification of Seifert
manifolds is given by

\begin{theorem}[\cite{NyR}]
\label{tma21}
Two Seifert symbols represent homeomorphic Seifert manifolds by a
fiber preserving homeomorphism if and only if one of the symbols can
be changed into the other by a finite sequence of the following moves: 

0. Permute the ratios.

1. Add or delete $\displaystyle{0\over1}$.

2. Replace the pair $\displaystyle{\beta_i\over\alpha_i},
{\beta_j\over\alpha_j}$ by
$\displaystyle{\beta_i+k\alpha_i\over\alpha_i},
{\beta_j-k\alpha_j\over\alpha_j}$.
\end{theorem}
\hfill$\Box$

If
$M=\displaystyle(Oo,g;{\beta_1/\alpha_1},\ldots,{\beta_t/\alpha_t})$,
then the number $e(M)=-\sum\frac{\beta_i}{\alpha_i}$ is called the
\emph{Euler number} of $M$.

\begin{theorem}[\cite{NyR}]
Let $M$ and $M'$ be Seifert manifolds with base spaces $F$ and $F'$,
respectively, and let $f:M\rightarrow M'$ be an orientation preserving and
fiber preserving map. If the degree of the induced map on a typical
fiber is $n$, and the degree of the induced map $\bar f:F\rightarrow
F'$ is $m$, then $e(M)=(m/n)e(M')$.
\end{theorem}
\hfill$\Box$

In particular if $\varphi:M\rightarrow M'$ is a branched covering
branched along fibers of $M'$, $e(M')=0$ if and only if $e(M)=0$.


\subsection{Basic Lemmas.}

\subsubsection{Factorization of coverings of Seifert manifolds.}
We show that any covering space of Seifert manifolds is
a product of 
simpler coverings.

\begin{lemma} 
  \label{lemma23}
  Assume that $G\leq S_n$ is a transitive group and
  $K\vartriangleleft G.$ 
  Assume that $K$ has $m$ orbits.
  Then there exist homomorphisms $q:G\to S_m$ and
  $\gamma:q^{-1}(St(1))\to S_{{n}/{m}}$ such that $St_G(1)\subset
  q^{-1}(St(1))$, and $q(K)=1$, and $\gamma(K)$ is transitive.
\end{lemma}
\begin{proof}
This follows by noting that the orbits of $K$ are a set of
imprimitivity blocks for $G$.
\end{proof}

\begin{corollary}
\label{coro24}
  Let  $M$ be a Seifert manifold, and
  let $\varphi:\widetilde{M}\to M$ be an~$n$-fold covering
  branched along fibers of $M$.
  Then there
  is a commutative diagram of coverings
  of Seifert manifolds branched along fibers
$$
\begin{tikzcd}
\widetilde M  \arrow[swap]{dd}{\varphi} \arrow{rd}{\varphi_\gamma} \\
  & N \arrow{ld}{\varphi_q} \\
M
\end{tikzcd}
$$
  such that $\varphi_q$ and $\varphi_\gamma$ are $m$-fold and
  $n/m$-fold branched coverings, respectively, and, if $h\subset M$ is
  an ordinary fiber, then
  $\varphi_q^{-1}(h)=h_1\sqcup\cdots\sqcup h_m$ with
  $\varphi_q|:h_i\rightarrow h$ a homeomorphism for $i=1,\dots,m$,  and
  $\varphi_\gamma^{-1}(\tilde h)$ is connected for $\tilde h$ any ordinary
  fiber of $N$.
  
\end{corollary}
\begin{proof}
  Recall that the subgroup
  $\langle h\rangle\lhd\pi_1(M)$. Then $\omega(\langle h\rangle)\lhd Image(\omega)$
  where $\omega:\pi_1(M-B_\varphi)\rightarrow S_n$ is the representation
  associated to $\varphi$, and~$B_\varphi$ is the branching of $\varphi$. Lemma~\ref{lemma23} applies.

\end{proof}

\begin{remark}
\label{rem25}
In the statement of Corollary~\ref{coro24}, if  $\omega_q$,
and $\omega_\gamma$ are the representations corresponding to
 $\varphi_q$, and $\varphi_\gamma$, respectively, the
conclusions of Corollary~\ref{coro24} are that $\omega_q(h)=(1)$, and
$\omega_\gamma(\tilde h)$ is a cycle of order~$n/m$ for $\tilde h$ an ordinary
  fiber of $N$.
\end{remark}


\subsubsection{Coverings of Seifert manifolds.}
Let $(F\times S^1,\{m_i\}_{i=1}^t)$ be a frame for the
Seifert
symbol~$\displaystyle(Oo,g;{\beta_1/\alpha_1},\dots,{\beta_t/\alpha_t})$, and  
let $\omega:\pi_1(F\times S^1)\rightarrow S_n$ be a representation,
and let $\varphi:\tilde M\rightarrow F\times S^1$ be the covering
space associated to $\omega$. If $\{\tilde m_j\}$ is a set of
components of $\varphi^{-1}(\cup m_i)$, one for each component of
$\varphi^{-1}\left(\bigcup_{i=1}^t (q_i\times S^1)\right)$, then
$(\tilde M,\{\tilde m_j\}_{j=1}^u)$ 
is a frame for some Seifert symbol $(Oo,\tilde g;{B_1/A_1},\dots,{B_u/A_u})$.

In this section we compute the numbers $\tilde g,u,A_1,B_1,\dots,A_u,B_u$
for some `generic' representations~$\omega:\pi_1(F\times
S^1)\rightarrow S_n$ (see Remark~\ref{rem25}). Proofs of these result
can be found in~\cite{trefoil}. 


Write $\varepsilon\in S_n$ for the $n$-cycle $\varepsilon=(1,2,\dots,n)$.

\begin{lemma}
\label{lemma26}
 Let $(F\times S^1,\{m_i\}_{i=1}^t)$ be a frame for 
$\displaystyle(Oo,g;{\beta_1/\alpha_1},\dots,{\beta_t/\alpha_t})$. 
Let~$\omega:\pi_1(F\times S^1)\rightarrow S_n$ be a representation such
that $\omega(h)=\varepsilon^s$ with $(n,s)=1$, and
$\omega(q_i)=\varepsilon^{r_i}$ for $i=1,\dots,t$ with $\sum r_i=0$. 
Let $s^*$ be any integer such that
 $s^*s\equiv1\pmod n$.

If
$\varphi:\tilde 
M\rightarrow F\times S^1$ is the covering associated to $\omega$ and
$\tilde m_i$ is a component of $\varphi^{-1}(m_i)$, $i=1,\dots,t$,
then $(\tilde M,\{\tilde m_i\}_{i=1}^t)$ is a frame for
$\displaystyle(Oo,g;{B_1/A_1},\dots,{B_t/ A_t})$, where 
$$A_i={n\alpha_i\over d_i}, \hskip 7ex B_i={\beta_i+s^*r_i\alpha_i\over d_i}$$
and $d_i=(n,\beta_i+s^*r_i\alpha_i)$,
$i=1,\dots,t$.
\end{lemma}
\hfill$\Box$


\begin{corollary}
\label{coro27}
 If $\sum r_i=0$, and $(n,s)=1$, and $s^*$ is any integer such that
 $s^*s\equiv1\pmod n$, then $\displaystyle(Oo,g;{B_1/ 
   A_1},\dots,{B_t/ A_t})$ is an $n$-fold branched covering of
 $\displaystyle(Oo,g;{\beta_1/\alpha_1},\dots,{\beta_t/
   \alpha_t})$, where 
$$A_i={n\alpha_i\over d_i}, \hskip 7ex B_i={\beta_i+s^*r_i\alpha_i\over d_i}$$
and $d_i=(n,\beta_i+s*r_i\alpha_i)$,
$i=1,\dots,t$.
\end{corollary}

\hfill$\Box$

\begin{lemma}
\label{lemma28}
Let $(F\times S^1,\{m_i\}_{i=1}^t)$ be a frame for
$\displaystyle(Oo,g;{\beta_1/ 1}, \dots, {\beta_u/ 1},\newline
{\beta_{u+1}/\alpha_{u+1}}, \dots,
{\beta_{u+t}/\alpha_{u+t}})$ with $\alpha_j>1$ for $j=u+1, \dots,
u+t$. Let~$\omega:\pi_1(F\times S^1)\rightarrow S_n$ be a transitive
representation such that $\omega(h)=(1)$, and
$\omega(q_i)=\varepsilon^{s_i}$ with $(n,s_i)=1$ for $i=1, \dots, u$, and
$\omega(q_j)$ is a product of $n_j$ disjoint cycles of order $d_j$ and has $k_j$
fixed points for $j=u+1, \dots, u+t$. 

If
$\varphi:\tilde M\rightarrow F\times S^1$ is the covering associated
to $\omega$ and $\{\tilde m_j\}_j$ is a set of components of
$\varphi^{-1}(\bigcup_{i=1}^{u+t}m_i)$, one for each component of  
$\varphi^{-1}(\bigcup_{i=1}^{u+t}(q_i\times S^1))$, then $(\tilde
M,\{\tilde m_j\}_j)$ is a frame for 

\noindent
$\displaystyle(Oo,\tilde g;
{n\beta_1/1}, \dots, {n\beta_u/1},$ \newline
$\displaystyle\overbrace{{\beta_{u+1}/\alpha_{u+1}}, \dots,
  {\beta_{u+1}/\alpha_{u+1}}}^{k_{u+1}\ \rm times},$ 
$\displaystyle\overbrace{{\beta_{u+1}/\alpha'_{u+1}}, \dots, {\beta_{u+1}/\alpha'_{u+1}}}^{n_{u+1}\ \rm times},
\dots,$ \newline
$\displaystyle\overbrace{{\beta_{u+t}/\alpha_{u+t}}, \dots,
  {\beta_{u+t}/\alpha_{u+t}}}^{k_{u+t}\ \rm times}, 
\overbrace{{\beta_{u+t}/\alpha'_{u+t}}, \dots, {\beta_{u+t}/\alpha'_{u+t}}}^{n_{u+t}\ \rm times})$,

\noindent where
$\alpha'_i=\alpha_i/ d_i$ for $i=u+1, \dots, u+t$, and\newline
$\tilde g=(2-n\chi(F)+u(n-1)+\sum n_i(d_i-1))/2$.
\end{lemma}
\hfill$\Box$


\begin{corollary}
\label{coro29}
With the numbers $d_i, n_i$ and $k_i$ as in  Lemma~2,

\noindent
$\displaystyle(Oo,\tilde g;$
$\displaystyle{n\beta_1/1}, \dots, {n\beta_u/1},$\newline
$\displaystyle\overbrace{{\beta_{u+1}/\alpha_{u+1}}, \dots, {\beta_{u+1}/\alpha_{u+1}}}^{k_{u+1}\ \rm times},$
$\displaystyle\overbrace{{\beta_{u+1}/\alpha'_{u+1}}, \dots, {\beta_{u+1}/\alpha'_{u+1}}}^{n_{u+1}\ \rm times},
\dots,$\newline
$\displaystyle\overbrace{{\beta_{u+t}/\alpha_{u+t}}, \dots, {\beta_{u+t}/\alpha_{u+t}}}^{k_{u+t}\ \rm times},$
$\displaystyle\overbrace{{\beta_{u+t}/\alpha'_{u+t}}, \dots, {\beta_{u+t}/\alpha'_{u+t}}}^{n_{u+t}\ \rm times})$

\noindent
is an $n$-fold branched covering of

\noindent
$\displaystyle(Oo,g;{\beta_1/1}, \dots, {\beta_u/1},
{\beta_{u+1}/\alpha_{u+1}}, \dots, {\beta_{u+t}/\alpha_{u+t}})$ 

\noindent
where $\alpha_i>1$, and
$\alpha'_i=\alpha_i/ d_i$ for $i=u+1, \dots, u+t$, and \newline
$\tilde g=(2-n\chi(F)+u(n-1)+\sum n_i(d_i-1))/2$.
\end{corollary}
\hfill$\Box$

\begin{remark}
\label{rem210}
In \cite{trefoil}, the statements of Lemma~\ref{lemma26} and Lemma~\ref{lemma28}
assume that the generators of the fundamental group of the orbit
surface are sent to the identity by the representation. Although we
only need that kind of covers in this paper, the hypothesis on the
generators of the surface is not needed for the conclusions of the lemmas.
\end{remark}

\subsubsection{Fiberings of $S^2\times S^1$.}
\label{sec223}
\begin{theorem}
Let $M$ be a Seifert manifold with at most two exceptional
fibers and base the 2-sphere. Then $M\cong S^2\times S^1$ if and only
if~$e(M)=0$. 
\hfill$\Box$
\end{theorem}

It follows that,

\begin{corollary}
Let $\alpha,\beta$ be coprime integers such that
$0\leq\beta\leq\alpha/2$. Then the Seifert fiberings
$(Oo,0;\beta/\alpha,-\beta/\alpha)$ are all the
Seifert fiberings of~$S^2\times S^1$. 
\hfill$\Box$
\end{corollary}


\section{Abelian coverings of $S^2\times S^1$.}
\label{sec3}
Let $t_{\alpha,\beta}$ be an ordinary fiber of
$(Oo,0;\beta/\alpha,-\beta/\alpha)$. 
We compute the Seifert invariants of the Abelian coverings of
$S^2\times S^1$ branched along~$t_{\alpha,\beta}$.

Write 
$M_0=S^2\times S^1-t_{\alpha,\beta}$.
Then we have the presentation
\begin{align*}
\pi_1(M_0)=\langle q_0,q_1,q_2,h:q_1^\alpha h^\beta=1,q_2^\alpha
h^{-\beta}&=1,q_0q_1q_2=1,\\
[q_i,h]&=1, i=0,1,2\rangle.\\
\end{align*}

Let $(F\times S^1,\{q_0,q_1^\alpha
h^\beta,q_1^\alpha h^{-\beta})$ be a frame for
$(Oo,0;0/1,\beta/\alpha,-\beta/\alpha)$.
Then an Abelian representation $\omega:\pi_1(F\times S^1)\rightarrow
S_n$ gives a 
branched covering of $(S^2\times S^1,t_{\alpha,\beta})$ if and only if
$\omega(q_1^\alpha h^\beta)=\omega(q_1^\alpha
h^{-\beta})=\omega(q_0q_1q_2)=(1)$.
Since $H_1(M_0)\cong \mathbb{Z}\oplus\mathbb{Z}_\alpha$, the image of
$\omega$ is cyclic or the sum of two cyclic groups. In
Section~\ref{sec32} we describe these image groups.

\subsection{The coverings}
Let $(F\times S^1,\{q_0,q_1^\alpha
h^\beta,q_1^\alpha h^{-\beta}\})$ be a frame for
$(Oo,0;0/1,\beta/\alpha,-\beta/\alpha)$.
Let $\omega:\pi_1(F\times S^1)\rightarrow S_n$ be a transitive representation
with Abelian image and $\omega(q_1^\alpha h^\beta)=\omega(q_1^\alpha
h^{-\beta})=\omega(q_0q_1q_2)=(1)$, and let $\varphi:\tilde M\rightarrow
(S^2\times S^1,t_{\alpha,\beta})$ be the associated branched covering. As in Remark~\ref{rem25}, $\varphi=\varphi_2\circ\varphi_1$ with
$\varphi_1:\tilde M\rightarrow N$ and $\varphi_2:N\rightarrow
S^2\times S^1$ branched coverings with corresponding
representations $\omega_1$ and~$\omega_2$, and~$\omega_1(\tilde h)$
one cycle and $\omega_2(h)=(1)$. 

We assume that $\omega(h)=(1)$ or
that~$\omega(h)$ is a cycle of order $n$. 

Write $\omega(q_1)=\sigma_1$, and $\omega(q_2)=\sigma_2$. 

\subsubsection{First case: $\omega(h)=(1)$.}
Since $\omega(q_i)^\alpha=\sigma_i^\alpha=(1)$, then the orders $o(\sigma_i)=a_i$ satisfy $a_i|\alpha$.

\emph{Case 1.1}: $a_1=1$; that is, $\sigma_1=(1)$. Then $\sigma_2$ is an $a_2$-cycle, and~$n=a_2$, and $\omega(q_0)=\sigma_2^{-1}$. Write $\alpha'=\alpha/a_2$. By Lemma~\ref{lemma28}, 
$$\tilde M=(Oo,0;\underbrace{\beta/\alpha,\dots, \beta/\alpha}_{a_2 \textrm{ times}},-\beta/\alpha'),$$
and $H_1(\tilde M)\cong \mathbb{Z}\oplus
\mathbb{Z}_{\alpha'}\oplus\bigoplus_{j=1}^{a_2-2}\mathbb{Z}_\alpha$. 

\emph{Case 1.2}: $a_2=1$; that is, $\sigma_2=(1)$. Then $\sigma_1$ is an $a_1$-cycle, and $n=a_1$, and $\omega(q_0)=\sigma_1^{-1}$. Write $\alpha'=\alpha/a_1$. By Lemma~\ref{lemma28}, 
$$\tilde M=(Oo,0;\beta/\alpha',\underbrace{-\beta/\alpha,\dots, -\beta/\alpha}_{a_1 \textrm{ times}}),$$
and $H_1(\tilde M)\cong \mathbb{Z}\oplus \mathbb{Z}_{\alpha'}\oplus\bigoplus_{j=1}^{a_1-2}\mathbb{Z}_\alpha$. 

\emph{Case 1.3}: $a_1,a_2>1$. The pair $\sigma_1,\sigma_2$ has a 4-tuple $(a_1,a_2,\delta,i_0)$, as in Section~\ref{sec32}, with $\delta|(a_1,a_2)$, and $0\leq i_0<\delta$, and $(i_0,\delta)=1$. Then~$n=a_1a_2/\delta$, and
\begin{itemize}
\item $\sigma_1$ is a product of $a_2/\delta$ cycles of order $a_1$.

\item $\sigma_2$ is a product of $a_1/\delta$ cycles of order $a_2$.

\item Write $\mu=(a_1,a_2/\delta+a_1i_0/\delta)$. Then the product $\sigma_1\sigma_2$ is a product of $ a_1a_2/(\mu\delta)$ cycles of order $\mu$. 
\end{itemize}
Write $\alpha_1=\alpha/a_1$, and $\alpha_2=\alpha/a_2$. By Lema~\ref{lemma28},
$$\tilde M=(Oo,g;\underbrace{\beta/\alpha_1,\dots,\beta/\alpha_1}_{a_2/\delta\textrm{ times}},\underbrace{-\beta/\alpha_2,\dots,-\beta/\alpha_2}_{a_1/\delta\textrm{ times}})$$
where 
$$g=1+\frac{(a_1-1)(a_2-1)-1}{2\delta}-\frac{\mu}{2},$$
and $$H_1(\tilde M)\cong
    \mathbb{Z}^{2g+1}\oplus\bigoplus_{i=3}^{(a_1+a_2)/\delta}\mathbb{Z}_{d_i}$$   
    where
    $$d_i=
    \begin{cases}
      {\alpha}/{(a_1,a_2)}, & \textrm{if } i\leq a_1/\delta,a_2/\delta,\\
      \alpha_1, &  \textrm{if }  a_1/\delta< i\leq a_2/\delta,\\
      \alpha_2, &  \textrm{if }  a_2/\delta< i\leq a_1/\delta,\\
      {\alpha}/{[a_1,a_2]}, &  \textrm{if }  a_1/\delta,a_2/\delta<i.\\
      \end{cases}
      $$ 

\begin{remark}
The genus $g$ of the orbit surface of $\tilde M$ is zero if and only if either $\delta=a_1=a_2$, or $\delta=1$ and $a_1=a_2=2$.

Indeed if $g=0$, then
$$2\delta+(a_1-1)(a_2-1)-1=\delta\mu\leq a_1,$$
and then
$$a_2-1\leq\frac{a_1-2}{a_1-1}\delta+\frac{1}{a_1-1}<\delta+1.$$
Therefore $a_2\leq \delta+1$. Since $\delta|a_2$ and $\delta\leq a_2$, it follows that either $a_2=\delta$, or $\delta=1$ and $a_2=2$.

Since the pair $\sigma_2,\sigma_1$ in reverse order has a 4-tuple
$(a_2,a_1,\delta,j_0)$, as in Section~\ref{sec32}, with $\delta|(a_1,a_2)$, and
$0\leq j_0<\delta$, and $(j_0,\delta)=1$, we see that
$\mu=(a_2,a_1/\delta+a_2j_0/\delta)$, and it follows, as above, that
either $a_1=\delta$, or $\delta=1$ and $a_1=2$. 

Then, either
\begin{itemize}
\item $\delta=a_1=a_2$, and $\varphi$ is an unbranched $a_1$-fold
  cyclic covering of $S^2\times S^1$, or 

\item $\delta=1$ and $a_1=a_2=2$. In this case, up to conjugation, $\sigma_1=(1,2)(3,4)$, and $\sigma_2=(1,3)(2,4)$, and, if $\alpha'=\alpha/2$, then
$\tilde M=(Oo,0;\beta/\alpha',\beta/\alpha',-\beta/\alpha',-\beta/\alpha').$
\end{itemize}

\end{remark}

\subsubsection{Second case: $\omega(h)=(1,2,\dots,n)=\varepsilon\in S_n$.}
Write $\omega(q_1)=\varepsilon^{r_1}$ and $\omega(q_2)=\varepsilon^{r_2}$. Since $\omega(q_1^\alpha h^\beta)=\varepsilon^{\alpha r_1+\beta}=(1)$, and $\omega(q_2^\alpha h^{-\beta})=\varepsilon^{\alpha r_2-\beta}=(1)$, it follows that
$$r_1\alpha+\beta\equiv0\mod{n}$$
and
$$r_2\alpha-\beta\equiv0\mod{n}.$$
Since $(\alpha,\beta)=1$, this two equations have solution if and only
if $(n,\alpha)=1$. Let $\alpha^*$ be an integer such that
$\alpha^*\alpha\equiv1\mod n$, and write $r_1=-\alpha^*\beta$ and
$r_2=\alpha^*\beta$; then $r_1+r_2=0$, and $r_1,r_2$ are solutions of
the system, if $(n,\alpha)=1$. Say $r_1\alpha+\beta=kn$; then
$r_2\alpha-\beta=-kn$. 

Since $\omega(q_0)=\omega((q_1q_2)^{-1})=\varepsilon^{-(r_1+r_2)}=(1)$, then $\varphi$ is an unbranched cyclic covering of $S^2\times S^1$.

Write $A_1=n\alpha/d_1$, $B_1=(\beta+r_1\alpha)/d_1$, and
$A_2=n\alpha/d_2$, $B_2=(-\beta+r_2\alpha)/d_2$ where
$d_1=(n,\beta+r_1\alpha)$ and $d_2=(n,-\beta+r_2\alpha)$. Then, by
Lemma~\ref{lemma26}, $\tilde M=(Oo,0;B_1/A_1,B_2/A_2)$. Notice that
$d_1=(n,\beta+r_1\alpha)=(n,kn)=n$. Then $A_1=\alpha$, and
$B_1=k$. Similarly $A_2=\alpha$, and $B_2=-k$. Then
$$\tilde M=(Oo,0;k/\alpha,-k/\alpha).$$
and $\tilde M\cong S^2\times S^1$.

\begin{theorem}
Let $\alpha,\beta$ be coprime integers with $0\leq\beta\leq\alpha/2$,
and let~$\omega:\pi_1(S^2\times S^1-t_{\alpha,\beta})\rightarrow S_n$
be a transitive representation with Abelian image.
Let $\varphi:\tilde
M\rightarrow(S^2\times S^1,t_{\alpha,\beta})$ be the
branched covering  associated with $\omega$. Let $h$ be an ordinary
fiber of $S^2\times S^1-t_{\alpha,\beta}$, and write the order
$o(\omega(q_i))=a_i$ where we write $a_i=1$ if $\omega(q_i)=(1)$. Also write $\alpha_i=\alpha/a_i$,
$i=1,2$.

\begin{enumerate}
\item If $\omega(h)=(1)$, then
  then $n=a_1a_2/\delta$ for some
    divisor $\delta|(a_1,a_2)$, and
    $$\tilde
    M=(Oo,g;\underbrace{\beta/\alpha_1,\dots,\beta/\alpha_1}_{a_2/\delta\textrm{
        times}}, 
    \underbrace{-\beta/\alpha_2,\dots,-\beta/\alpha_2}_{a_1/\delta\textrm{
        times}})$$ 
    where 
    $$g=1+\frac{(a_1-1)(a_2-1)-1}{2\delta}-\frac{\mu}{2},$$
    and $\mu=(a_1,a_2/\delta+a_1i_0/\delta)$ for some $i_0$ such that $(i_0,\delta)=1$,
    and $0\leq i_0<\delta$. Also

    $$H_1(\tilde M)\cong
    \mathbb{Z}^{2g+1}\oplus\bigoplus_{i=3}^{(a_1+a_2)/\delta}\mathbb{Z}_{d_i}$$   
    where
    $$d_i=
    \begin{cases}
      {\alpha}/{(a_1,a_2)}, & \textrm{if } i\leq a_1/\delta,a_2/\delta,\\
      \alpha_1, &  \textrm{if }  a_1/\delta< i\leq a_2/\delta,\\
      \alpha_2, &  \textrm{if }  a_2/\delta< i\leq a_1/\delta,\\
      {\alpha}/{[a_1,a_2]}, &  \textrm{if }  a_1/\delta,a_2/\delta<i.\\
      \end{cases}
      $$

      \noindent
      If $a_1,a_2>1$, then $g=0$ if and only if either $\delta=a_1=a_2$, 
      or $\delta=1$ and $a_1=a_2=2$. If $g=0$ and $\delta=a_1=a_2$, 
      then $\tilde M\cong S^2\times S^1$ and $\varphi$ is an unbranched 
      cyclic covering.

\item If $\omega(h)=(1,2,\dots,n)$, then $(n,\alpha)=1$, and $\varphi$
  is an unbranched covering space, and
  $$\tilde M=(Oo,0;k/\alpha,-k/\alpha)\cong S^2\times S^1$$
  where $k=(-\beta\alpha^*\alpha+\beta)/n$, and
  $\alpha^*\alpha\equiv1\pmod{n}$.

\end{enumerate}

\end{theorem}

\begin{corollary}
Let $\varphi:\tilde M\rightarrow(S^2\times S^1,t_{\alpha,\beta})$ be
an $n$-fold Abelian branched covering.
If $H_1(\tilde M)\cong\mathbb{Z}$, then $\tilde M\cong S^2\times S^1$
and $\varphi$ is an unbranched cyclic covering space.
\end{corollary}

\subsection{Abelian permutation groups generated by two
  elements.}
\label{sec32}
Let $\sigma_1,\sigma_2\in S_n$ be a pair of permutations such that the group $G=\langle\sigma_1,\sigma_2\rangle$ is Abelian and transitive.

We associate a 4-tuple of integers $(a_1,a_2,\delta,i_0)$ to the pair $\sigma_1,\sigma_2$.

Assume that the order $o(\sigma_i)=a_i$ where we write $a_i=1$, if $\sigma_i=(1)$ $(i=1,2)$. Then $n=o(G)=a_1a_2/\delta$ for some divisor $\delta|(a_1,a_2)$. 

Let $T$ be a torus with a CW structure with one vertex $v$, two
1-cells, $\gamma_1,\gamma_2$ and one 2-cell $d$. Since $\sigma_1$ and
$\sigma_2$ commute, the assignment~$\omega:\pi_1(T)\rightarrow S_n$
such that $\omega(\gamma_i)=\sigma_i$ defines a representation. Let~$\varphi:\tilde T\rightarrow T$ be the covering space associated to
$\omega$. Then $\tilde T$ admits a CW structure with 
cells the liftings of the cells of $T$.


The 1-skeleton of this CW structure of $\tilde T$ gives a portrait of
$\sigma_1$ and~$\sigma_2$: Draw a grid of $a_2/\delta$ horizontal,
parallel lines with $a_1+1$ equally spaced vertices each, so that,
after identifying the ends of a line, we obtain a circle with $a_1$
vertices, and complete the grid with $a_1$ vertical lines with
$a_2/\delta+1$ vertices. Now the top and the bottom of the grid are
identified in such a way that the vertical lines become $a_1/\delta$
circles with~$a_2$ vertices each: that is, $\delta$ equally spaced
vertical lines become a circle.  

Then for some number $i_0$ such that $(\delta, i_0)=1$, and $0\leq
i_0<\delta$, after identifying the sides of the grid with the
identity, the top of the grid is identified with the bottom with a
twist that identifies the top of the first vertical line with the
bottom of the $(i_0\delta)$-th vertical line. 

The pair $\sigma_1,\sigma_2$ can be recovered from the 4-tuple
$(a_1,a_2,\delta,i_0)$, up to conjugation. Notice that if we reverse
the order of $\sigma_1,\sigma_2$, the pair~$\sigma_2,\sigma_1$ has a 4-tuple of the form $(a_1,a_2,\delta,j_0)$ with
$(\delta, j_0)=1$, and $0\leq j_0<\delta$ (usually $i_0\neq j_0$). 

The product $\sigma_1\sigma_2$ is drawn on $\tilde T$ as the union of
the liftings of the diagonal of the 2-cell $d$ of $T$ which is
homotopic to $\gamma_1\gamma_2$. Therefore the order 
$$o(\sigma_1\sigma_2)=\frac{a_1a_2/\delta}{(a_1,(i_0a_1+a_2)/\delta)}.$$

\section{Universals.}
\label{sec4}
Write $|X|$ for the number of components of the space $X$.
In this section we fix integers $\alpha,\beta$ such that
$(\alpha,\beta)=1$, $0\leq\beta\leq\alpha/2$. 
We write $t_{\alpha,\beta}\subset S^2\times S^1$ for  the torus knot
which is an ordinary fiber of~$(Oo,0;\beta/\alpha,-\beta/\alpha)$.

\begin{lemma}
\label{lemma41}
There exists $\varphi:(Oo,0;)\rightarrow
(S^2\times S^1,t_{\alpha,\beta})$ a $2\alpha$-fold branched
covering such that $|\varphi^{-1}(t_{\alpha,\beta})|=2\alpha-2$.
\end{lemma}
\begin{proof}
Let $(F\times S^1,\{m_0,m_1,m_2\})$ be a frame for
$(Oo,0;0/1,\beta/\alpha,-\beta/\alpha)$, and let $\omega:F\times
S^1\rightarrow S_{2\alpha}$ be the representation such that
\begin{align*}
\omega(h)&=(1)\\
\omega(q_0)&=(\alpha-1,2\alpha-1)(\alpha,2\alpha)\\
\omega(q_1)&=(1,2,\dots,\alpha-2,\alpha-1,2\alpha)(\alpha,\alpha+1,\alpha+2,\dots,2\alpha-2,2\alpha-1)\\
\omega(q_2)&=(2\alpha,2\alpha-1,\dots,\alpha+2,\alpha+1)(\alpha,\alpha-1,\dots,2,1)\\
\end{align*}
Notice that $\omega(m_1)=\omega(q_1^\alpha h^\beta)=(1)$, and
$\omega(m_2)=\omega(q_1^\alpha h^{-\beta})=(1)$, and
$\omega(q_0q_1q_2)=(1)$. By Lemma~\ref{lemma28}, the completion of the
covering associated with $\omega$ is
$$\varphi:(Oo,0;\frac\beta1,\frac\beta1,\frac{-\beta}1,\frac{-\beta}1)\rightarrow (Oo,0;\frac01,\frac\beta\alpha,\frac{-\beta}\alpha),$$
and is branched along the $0/1$-fiber, that is, is branched along
$t_{\alpha,\beta}$. Since $\omega(q_0)$ has $2\alpha-4$ fixed points
and two 2-cycles, therefore $|\varphi^{-1}(t_{\alpha,\beta})|=2\alpha-2$.
\end{proof}

\begin{lemma}
\label{lemma42}
If $k>0$ and $\alpha\geq3$, then there exists $\varphi:(Oo,0;)\rightarrow
(S^2\times S^1,t_{\alpha,\beta})$ a branched
covering such that $|\varphi^{-1}(t_{\alpha,\beta})|\geq k$.
\end{lemma}
\begin{proof}
By previous lemma, there is $\psi:(Oo,0;)\rightarrow
(S^2\times S^1,t_{\alpha,\beta})$ with
$|\psi^{-1}(t_{\alpha,\beta})|=2\alpha-2\geq4$. Let $(F\times
S^1,\{m_1,m_2\})$ be a frame for $(Oo,0;0/1,0/1)$ where the
$0/1$-fibers $e_1,e_2$ are contained in $\psi^{-1}(t_{\alpha,\beta})$. Let $\omega:F\times
S^1\rightarrow S_k$ be the representation such that 
$\omega(h)=(1)$,
$\omega(q_1)=\varepsilon$, and
$\omega(q_2)=\varepsilon^{-1}$.
where $\varepsilon=(1,2,\dots,k)\in S_k$. By Lemma~\ref{lemma28}, the completion of the
covering associated with $\omega$ is
$$\varphi:(Oo,0;\frac01,\frac01)\rightarrow (Oo,0;)$$
and is branched along $e_1\cup e_2$. Let $e_3\subset
\psi^{-1}(t_{\alpha,\beta})$ be a fiber different from $e_1$ and
$e_2$. Then $|\varphi^{-1}(e_3)|=k$, and the lemma follows.
\end{proof}

\begin{lemma}
\label{lemma43}
Let $g\geq0$, and $k\geq 2g+2$ be integers, and assume $\alpha\geq3$.
Then there exists a branched
covering $\varphi:(Oo,g;)\rightarrow
(S^2\times S^1,t_{\alpha,\beta})$ such that
$|\varphi^{-1}(t_{\alpha,\beta})|\geq k$.
\end{lemma}
\begin{proof}
By previous lemma, there is $\varphi:(Oo,0;)\rightarrow
(S^2\times S^1,t_{\alpha,\beta})$ a branched
covering such that $|\varphi^{-1}(t_{\alpha,\beta})|\geq k\geq 2g+2$. 
Then there are~$2g+2$ different fibers $e_1,e_2,\dots,e_{2g+2}\subset |\varphi^{-1}(t_{\alpha,\beta})|$. The double cyclic covering of~$(Oo,0;)$
branched along $e_1\cup e_2\cup\dots\cup e_{2g+2}$ is $(Oo,g;)$ (the double branched covering, in
terms of a representation $\omega$ of a frame for $(Oo,0;0/1,\dots,0/1)$ into~$S_2$, is~$\omega(h)=(1)$, and
$\omega(q_i)=(1,2)$ for~$i=1,2,\dots,2g+2$).
\end{proof}

\begin{lemma}
\label{lemma44}
Let $g\geq0$ and $k>0$ be integers, and assume $\alpha\geq3$. If~$b/a$ is
a reduced rational number, then there exists a branched
covering~$\varphi:(Oo,g;b/a,-b/a)\rightarrow
(S^2\times S^1,t_{\alpha,\beta})$ such that
$|\varphi^{-1}(t_{\alpha,\beta})|\geq k$
and~$\varphi^{-1}(t_{\alpha,\beta})$ contains the fibers of the ratios
of $(Oo,g;b/a,-b/a)$.
\end{lemma}
\begin{proof}
By previous lemma, there is a branched
covering $\psi:(Oo,g;)\rightarrow
(S^2\times S^1,t_{\alpha,\beta})$ with
$|\psi^{-1}(t_{\alpha,\beta})|\geq k$. Let $(F\times S^1,\{m_0,m_1,m_2,m_3\})$ be a frame for
$(Oo,g;0/1,0/1,b/1,-b/1)$, and let $\omega:F\times
S^1\rightarrow S_a$ be the representation such that 
\begin{align*}
\omega(h)&=\varepsilon^{b^*}\\
\omega(q_0)&=\omega(q_2)=\varepsilon^{-1}\\
\omega(q_1)&=\omega(q_3)=\varepsilon\\
\end{align*}
where $b^*b\equiv1\pmod{a}$, and $\varepsilon=(1,2,\dots,a)\in S_a$.
By Lemma~\ref{lemma28}, the completion of the
covering associated with $\omega$ is
$$\varphi:(Oo,g;\frac ba,\frac{-b}a,\frac01,\frac01)\rightarrow (Oo,g;)$$
and is branched along the fibers of the ratios of $(Oo,g;0/1,0/1,b/1,-b/1)$.
\end{proof}

\begin{theorem}
\label{thm45} Assume that $\alpha\geq3$.
If $\sum_{i=1}^t\beta_1/\alpha_i=0$, then there exists
a branched
covering 
$\varphi:(Oo,g;\beta_1/\alpha_1,\dots,\beta_t/\alpha_t)\rightarrow(S^2\times S^1,t_{\alpha,\beta})$   such that~$\varphi^{-1}(t_{\alpha,\beta})$ contains the fibers of the ratios 
of 
\newline
$(Oo,g;\beta_1/\alpha_1,\dots,\beta_t/\alpha_t)$ and some extra
$0/1$ fibers.
\end{theorem}
\begin{proof}
Induction on $t$. Assume that $t=2$. If $\alpha_1=1$, by Lemma~\ref{lemma43}, 
$(Oo,g;\beta_1/1,-\beta_1/1)$ is a branched covering of $(S^2\times S^1,t_{\alpha,\beta})$; otherwise, by Lemma~\ref{lemma44}, $(Oo,g;\beta_1/\alpha_1,-\beta_1/\alpha_1)$ is a branched covering of $(S^2\times S^1,t_{\alpha,\beta})$ as needed. 

Assume 
that the theorem is true for $t\geq2$, that is, if the sum of the ratios
$\sum_{i=1}^t\beta_i/\alpha_i=0$, then
the manifold $(Oo,g;\beta_1/\alpha_1,\dots,\beta_t/\alpha_t)$ is a branched
covering of $(S^2\times S^1,t_{\alpha,\beta})$ with a branched
covering $\varphi$ such that~$\varphi^{-1}(t_{\alpha,\beta})$ contains the fibers of the ratios 
of the manifold~$(Oo,g;\beta_1/\alpha_1,\dots,\beta_t/\alpha_t)$, plus some extra
$0/1$ fibers. 
Assume now that
$\sum_{i=1}^t\beta_i/\alpha_i+u/v=0$. 

Let $(F\times
S^1,\{m_0,m_1,\dots,m_t\})$ be a frame for the manifold 
\newline
$(Oo,g;u/1,b_1/a_1,\dots,b_t/a_t)$ 
where the $b_i/a_i$ is the reduced ratio $v\beta_i/\alpha_i$;
that is, $b_i=v\beta_i/(v,\alpha_i)$, and
$a_i=\alpha_i/(v,\alpha_i)$ for $i=1,\dots,t$. Let
$\omega:\pi_1(F\times S^1)\rightarrow S_v$ be the representation
such that $\omega(h)=(1,2,\dots,v)$, and $\omega(q_i)=(1)$. Here
we have $s=1$ and $r_i=0$. Then, by Lemma~\ref{lemma26}, the completion of the
covering associated with $\omega$ is
$$\varphi:(Oo,g,\frac BA,\frac{B_1}{A_1},\dots,\frac{B_t}{A_t})\rightarrow
(Oo,g;\frac\beta1,\frac{b_1}{a_1},\dots,\frac{b_t}{a_t})$$
where $d=(v,u)=1$, and $B=u/d=u$, and
$A=v/d=v$, and for $i\in\{1,\dots,t\}$,
$d_i= (v,b_i)= (v,v\beta_i/(v,\alpha_i))=
v/(v,\alpha_i)$, and
$B_i= b_i/d_i= (v\beta_i/(v,\alpha_i))/(v/(v,\alpha_i))
=\beta_i$, 
and $A_i=v a_i/d_i=(v
\alpha_i/(v,\alpha_i))/(v/(v,\alpha_i))=\alpha_i$. That
is, $(Oo,g;u/v,\beta_1/\alpha_1,\dots,\beta_t/\alpha_t)$ is a
branched covering of $(Oo,g;u/1,b_1/a_1,\dots,b_t/a_t)$ with
branching along the fibers of the ratios. 

Since
$u/1+\sum_{i=1}^tb_i/a_i=v(u/v+\sum_{i=1}^t\beta_i/\alpha_i)=0$,
\newline
and $(Oo,g;u/1,b_1/a_1,\dots,b_t/a_t)$ is isomorphic 
\newline
to
$(Oo,g;(u a_1+b_1)/a_1,b_2/a_2,\dots,b_t/a_t)$,
by induction, 
the manifold
$(Oo,g;u/1,b_1/a_1,\dots,b_t/a_t)$ is a branched
covering of $(S^2\times S^1,t_{\alpha,\beta})$ with a branched
covering $\psi$ such that~$\psi^{-1}(t_{\alpha,\beta})$ contains the fibers of the ratios 
of $(Oo,g;u/1,b_1/a_1,\dots,b_t/a_t)$ plus some extra 0/1 fibers. 
We conclude that 
$(Oo,g;u/v,\beta_1/\alpha_1,\dots,\beta_t/\alpha_t)$ is also
a branched covering of $(S^2\times S^1,t_{\alpha,\beta})$ as required.

\end{proof}

\begin{corollary}
\label{coro46}
Let $M$ be an orientable Seifert manifold with orientable basis and $e(M)=0$, and assume that $\alpha\geq3$. Then $M$ is
a branched covering of $(S^2\times S^1,t_{\alpha,\beta})$.
\end{corollary}

\subsection{A non-trivial non-universal torus knot in $S^2\times
  S^1$.}
\begin{theorem}
\label{thm47}
If $M$ is a branched covering of $(S^2\times S^1,t_{2,1})$, then
$M\cong S^2\times S^1$.
\end{theorem}
\begin{proof}
Let $\varphi:M\rightarrow(S^2\times S^1,t_{2,1})$ be an $n$-fold branched
covering, and write $F=S^2\times\{1\}$. Then $F$ is a non-separating
2-sphere generating $H_2(S^2\times S^1)$, and $|F\cap t_{2,1}|=2$; thus $\varphi^{-1}(F)$ is a disjoint
union of 2-spheres. 
Now an orientation on~$F$ induces an orientation on~$\varphi^{-1}(F)$
through $\varphi$, and with this orientation, $\varphi^{-1}(F)$ defines a class in $H_2(M)$. since $\varphi|:\varphi^{-1}(F)\rightarrow F$ is an
$n$-fold branched covering, then 
$\varphi_*[\varphi^{-1}(F)]=n[F]\neq0$.
In particular $[\varphi^{-1}(F)]\in H_2(M)-\{0\}$,
%
and, therefore, some component of $\varphi^{-1}(F)$ is
non-separating. Then~$M$ is an orientable Seifert manifold with a
non-separating 2-sphere. We conclude that $M\cong S^2\times S^1$.
\end{proof}


\begin{thebibliography}{99}


\bibitem{Fox}
R. Fox. Covering spaces with singularities, in Lefschetz Symposium,
Princeton Math. Ser. 12 (1967), 55--66.

\bibitem{NyR}
W. Newman and F. Raymond. Seifert manifolds, plumbing $\mu$-invariant
and orientation reversing maps. Lecture Notes in Math., Vol.~664,
Springer Verlag, Berlin, 1978, pp. 163--196.

\bibitem{trefoil}
V. N\'u\~nez and E. Ram\'\i rez. The trefoil knot is as universal as
it can be. Topology Appl. 130 (2003), 1--17.

\bibitem{seifert}
H. Seifert. Topologie Dreidimensionaler Gefaserter Räume. Acta Math. 60 (1933), 147–-238.

\end{thebibliography}
\end{document}